\documentclass[a4paper, 12pt]{amsart}

\usepackage[french, english]{babel}
\usepackage[normalem]{ ulem }
\usepackage{soul}

\usepackage[T1]{fontenc}
\usepackage[utf8]{inputenc}
\usepackage{amsmath,amsfonts,amssymb,amsopn,amscd,amsthm}
\usepackage{gensymb}
\usepackage{tikz}
\usepackage{comment}
\usepackage{dsfont}
\usepackage{bm}
\usepackage{graphicx}
\usepackage{color}
\usepackage[colorlinks]{hyperref}
\usepackage{epigraph}
\usepackage{enumerate}
\usepackage{xargs}
\usepackage[prependcaption]{todonotes}

\usepackage{bigints}

\usepackage{marginnote}
\usepackage{todonotes}
\usepackage[left=3.5cm,top=2.5cm,bottom=3cm,right=4cm, marginparwidth=3cm]{geometry}

\title[ Free deconvolutions ]
      { Estimation of large covariance matrices via free deconvolution: computational and statistical aspects }

\author[Chhaibi]{\textsc{Reda Chhaibi}}
\address{Universit\'e Paul Sabatier, Toulouse 3 -- Institut de math\'ematiques de Toulouse (IMT) -- 118, route de Narbonne, 31400, Toulouse, France}
\email{{\tt reda.chhaibi@math.univ-toulouse.fr}}

\author[Gamboa]{\textsc{Fabrice Gamboa}}
\address{Universit\'e Paul Sabatier, Toulouse 3 -- Institut de math\'ematiques de Toulouse (IMT) -- 118, route de Narbonne, 31400, Toulouse, France}
\email{{\tt fabrice.gamboa@math.univ-toulouse.fr}}

\author[Kammoun]{\textsc{Slim Kammoun}}
\address{Universit\'e Paul Sabatier, Toulouse 3 -- Institut de math\'ematiques de Toulouse (IMT) -- 118, route de Narbonne, 31400, Toulouse, France}
\email{{\tt slim.kammoun@math.univ-toulouse.fr}}

\author[Velasco]{\textsc{Mauricio Velasco}}
\address{Departamento de Matem\'aticas, H-304 Universidad de los Andes, Bogot\'a, Colombia}
\email{{\tt mvelasco@uniandes.edu.co}}

\date{\today}

\allowdisplaybreaks[4]

\DeclareMathOperator{\tr}{tr}

\DeclareMathOperator{\Card}{ \textrm{Card} }

\DeclareMathOperator{\Var}{Var}

\DeclareMathOperator{\Spec}{Spec}

\DeclareMathOperator{\diag}{diag}
\DeclareMathOperator{\id}{id}
\DeclareMathOperator{\Tr}{Tr}
\newcommand{\boxdiag}{{\ \frame{\tikz{\draw[line width=0.7pt,line cap=round] (0,0) -- (6.5pt,6.5pt);}}\;}}

\def\half{\frac{1}{2}}

\def\1{{\mathbf 1}}

\def\N{{\mathbb N}}

\def\R{{\mathbb R}}
\def\C{{\mathbb C}}
\def\S{{\mathbb S}}

\def\W{\mathbb{W}}

\def\S{\mathbb{S}}
\def\P{{\mathbb P}}
\def\E{{\mathbb E}}

\def\X{{\mathbb X}}
\def\Y{{\mathbb Y}}

\def\Cc{{\mathcal C}}

\def\Lc{{\mathcal L}}
\def\Mc{{\mathcal M}}
\def\Nc{{\mathcal N}}
\def\Oc{{\mathcal O}}

\newcommand{\CC}{\mathbb{C}}
\newcommand{\RR}{\mathbb{R}}

\newtheorem{thm}{Theorem}[section]
\newtheorem{proposition}[thm]{Proposition}

\newtheorem{question}[thm]{Question}

\newtheorem{definition}[thm]{Definition}

\newtheorem{lemma}[thm]{Lemma}

\newtheorem{remark}[thm]{Remark}
\newtheorem{assumption}[thm]{Assumption}

\numberwithin{equation}{section}
\numberwithin{figure}{section}

\begin{document}

\begin{abstract}
The estimation of large covariance matrices has a high dimensional bias. Correcting for this bias can be reformulated via the tool of Free Probability Theory as a free deconvolution.

The goal of this work is a computational and statistical resolution of this problem. Our approach is based on complex-analytic methods methods to  invert $S$-transforms. In particular, one needs a theoretical understanding of the Riemann surfaces where multivalued $S$ transforms live and an efficient computational scheme.
\end{abstract}

\renewcommand{\subjclassname}{%
  \textup{2010} Mathematics Subject Classification}

\maketitle

\setcounter{tocdepth}{2}
\medskip
\hrule
\tableofcontents
\hrule

\section{Introduction}

{\bf Estimating large sample covariance matrices: } An excellent paper explaining the intricacies of the problem is \cite{elkaroui2008}. Recall that given an i.i.d. sample $X_1, \dots, X_n$ of random vectors in $\R^d$, then the classical estimators of the mean and covariance are respectively given by:
$$ \hat{\mu}    = \frac{1}{n} \sum_{i=1}^n X_i \in \R^d$$
$$ \hat{\Sigma} = \frac{1}{n-d} \sum_{i=1}^n \left( X_i - \hat{\mu} \right) \left( X_i - \hat{\mu} \right)^* \in M_d(\R) \ .$$

Such estimators are perfectly well-behaved for fixed $d$ as $n \rightarrow \infty$. Nevertheless, if $d$ and $n$ are of comparable size, the situation changes dramatically. Indeed, suppose now that $d = \lfloor cn \rfloor$ as $n \rightarrow \infty$ . For convenience, we assume that the $X_i$'s have i.i.d. reduced and centered components so that $\mu=\hat{\mu}=0$ and the estimator of the covariance can be taken as:
$$ \hat{\Sigma} = \frac{1}{n} \sum_{i=1}^n X_i X_i^*
                = \frac{1}{n} \X \X^* \ ,$$
where $\X = \left( X_1, \dots, X_n \right) \in M_{d,n}(\R)$ is the matrix with columns given by the $X_i$'s.

Naively, one would think that the empirical spectral distribution:
$$ \mu_n(dx) := \frac{1}{d} \sum_{\lambda \in \Spec(\frac{1}{n} \X \X^*)} \delta_{\lambda}(dx)$$
converges to $\delta_1(dx)$ -- because $\frac{1}{n} \X \X^*$ ought to converge to the identity as in the case of fixed $d$. However the truth is that:

\begin{thm}[Marcenko-Pastur]
\label{thm:MP}
Assume for simplicity that $c<1$. Almost surely, as $n \rightarrow \infty$, we have the weak convergence of probability measures:
$$ \lim_n \mu_n = MP_c$$
where $l=(1-\sqrt{c})^2$, $r=(1+\sqrt{c})^2$ and
$$ MP_c(dx) := \mathds{1}_{\{ x \in [l, r] \}}
               \frac{\sqrt{(x-l)(r-x)}}{2 \pi x}
               dx
$$
is the Marcenko-Pastur distribution.
\end{thm}

The Marcenko-Pastur being a spread-out distribution around $1$ clearly introduces a high dimensional bias in the estimation of (the spectrum) covariance matrices. 

\bigskip

{\bf The assumptions: } We assume that
\begin{align}
\label{eq:observation}
   \X = \Sigma^\half \Y \ ,
\end{align}
where $\Sigma \in M_d(\R)$ is the true covariance matrix and $\Y \in M_{d,n}(\R)$ has i.i.d. coefficients. The following is a natural setup in Random Matrix Theory although this can be relaxed.

\begin{assumption}
We assume the convergence of the spectral measure 
$$ \nu_d :=
   \frac{1}{d} \sum_{\lambda \in \Spec(\Sigma)}
   \delta_{\lambda}
$$
to probability measure $\nu  = \lim_{d \rightarrow \infty} \nu_d$.
\end{assumption}

This assumption can be seen as a sparsity assumption, which will allow to estimate our objects as the dimension grows. Indeed, it basically says that if one performs a PCA of the true (typically unknown) covariance matrix $\Sigma$ and observe the eigenmodes, an asymptotic shape will appear. The main question we address is
\begin{question}
\label{question:main}
How to construct a covariance matrix $\widehat{\Sigma}$ such that for all smooth $f$
$$
   \frac{1}{d} \Tr f\left( \widehat{\Sigma}\right)
   =
   \nu_d(f)   
$$
estimates $\nu(f)$ and fluctuates in $\frac{1}{n}$.
\end{question}

The speed $\frac{1}{n}$ is the natural speed of convergence for linear statistics in the context of RMT.

\subsection{Reformulation thanks to FPT}
Free probability Theory (FPT) has risen from the pioneering works of Voiculescu and established itself as the correct framework to handle the macroscopic behavior of large random matrices. It gives a deterministic model thanks to which one can compute the spectrum of large matrices related by the multiplicative relation in Eq. \eqref{eq:observation}.

We recall the definition of the various transforms which re-encode measures \cite{V87}. Given a measure $\mu \in \Mc_1(\R)$, we define the Stieljes transform as
$$ 
   \forall z \in \C \backslash \R, \ 
   G_\mu(z)
   :=
   \int_\R \frac{\mu(dt)}{z-t} \ .
$$
It is a standard fact that the knowledge of $\mu$ and the knowledge of $G_\mu$ are equivalent. Another re-encoding is the moment generating function:
$$ M_\mu(z) := zG_\mu(z) - 1 = \sum_{n=1}^\infty \frac{m_n(\mu)}{z^n} \ ,$$
which is invertible in the neighborhood of infinity provided $m_1(\mu) \neq 0$. As such the functional inverse $M_\mu^{\langle -1 \rangle}$ is well-defined and holomorphic on a neighborhood of zero. Then the Voiculescu $S$-transform is:
$$ S_\mu(m) :=  \frac{1+m}{m M_\mu^{\langle -1 \rangle}(m)} \ .$$
Because of holomorphic extension, all of $\mu$, $G_\mu$, $M_\mu$ and $S_\mu$ contain the same information albeit in different forms. As such, we loosely refer to them as re-encodings.

This allows a neat reformulation of the asymptotic relation between the observed spectrum of $\X$ and the unknown measure $\nu$.

\begin{thm}[Summary of FPT results for estimating covariance matrices]
\label{thm:summary}
Under our assumptions, $\mu_{\frac{1}{n}\X \X^*}$ converges weakly to a  measure $\mu$ which is a deterministic function of the measures of $\nu$ and $MP_c$, which are the limiting spectral measures of $\Sigma$ and $\Y \Y^*$. This measure is called the multiplicative free convolution:
$$ \mu = \nu \boxtimes MP_c \ .$$

Moreover, the above relation can be recast into two equivalent equations.

\begin{itemize}
\item ($S$-transforms) The $S$ transform satisfies:
$$ S_{\mu}(m) = S_\nu(m) S_{MP_c}(m) \ ,$$
for all $m$ in a neighborhood of $0$. 

\item (Marchenko-Pastur equation) 
$$
   -\frac{1}{G_\mu(z)}
   =
   z - c \int \frac{\lambda \nu(d\lambda)}
                   {1 + \lambda G_\mu(z)} \ .
$$ 
\end{itemize}

\end{thm}
\begin{proof}[Pointers to proof]
The convergence to the multiplicative free convolution is a particular case of Voiculescu's theorem \cite{V87} and the fact that $\Sigma$ and $\Y$ are asymptotically free. And the spectrum of $\Y \Y^*$ converges to the $MP_c$ measure following the classical theorem \ref{thm:MP}.

The equivalence between the two characterizations of free convolution is classical. The use of $S$-transforms is the Voiculescu's original point of view. The Marchenko-Pastur equation can be found in \cite[Theorem 1]{elkaroui2008} and is an instance of the more general subordination phenomenon \cite{BB07}.
\end{proof}

In order to stress that we are interested in free deconvolution, we introduce the self-explanatory symbol $\boxdiag$, that is implicitly defined by:
$$
\mu_3 = \mu_1 \boxtimes \mu_2
\ \Longleftrightarrow \ 
\mu_1 = \mu_3 \boxdiag \mu_2 \ .
$$
Basically, for {\it input measures} $\mu$ and $MP_c$, the goal is to compute the {\it output measure} $\nu = \mu \boxdiag MP_c$. Borrowing the terminology from signal processing, this is the free deconvolution of the measure $\mu$ by the Marchenko-Pastur distribution $MP_c$.

More precisely, this has to be done in {\it empirically}. Indeed, one only observes the empirical spectral measure $\mu_n$ of $\frac1n \X \X^*$. Then we aim at constructing an estimator $\widehat{\nu}_n$ such that:
$$ 
   S_{\widehat{\nu}_n} \approx S_{\widehat{\mu}_n} / S_{MP_c} \ .
$$
Although at this point the meaning of $\approx$ remains loose, let us say the following. The symbol $\approx$ has to be understood as a proximity of holomorphic functions at the neighborhood of zero. In turn, this reflects proximity of the free-cumulants of the underlying measures and thus the usual weak topology on measures.

\subsection{Literature review}
There are very few papers that deal with the statistical and computational aspects of free probability. Two notable references stand out and need to be compared to this work.

Among the first papers on the topic, there is \cite{BD08}. Notice in particular the comment in Section VI regarding the use of methods based on the $R$-transform: "Unfortunately, this method is interesting only in very few cases, because the operations which are necessary here (the inversion of certain functions, extension of analytic functions) are almost always impossible to realize practically." In fact, we are going against the grain by doing exactly that. The inversion of analytic functions can be done very effectively via homotopy methods and the Newton-Raphson scheme, on the condition of carefully controlling coverings and  basins of attraction.

The pioneering paper \cite{elkaroui2008} solves a convex optimization problem coming from the Marchenko-Pastur equation of Theorem \ref{thm:summary}. Basically, one looks for the measures $\nu$ minimizing the error in that equation. Clearly, there is a stability issue: the approximate zero of an equation is not necessarily close to the zero of that equation.

More recently, in a series of papers \cite{ATV20, tarrago20}, the authors start by computing Stieljes transforms of the measure of interest using a fixed point algorithm. This is a side-product of subordination method -- See \cite{BB07, BMS17} and the references therein. In any case, although subordination is a great and flexible tool, one can pinpoint two major hurdles:
\begin{itemize}
\item subordination is a fixed-point method. And fixed-point methods always loose to the celebrated Newton-Raphson scheme in small dimensions.
\item the Stieljes transforms are computed away from the real axis and therefore the output of free deconvolution $\mu \boxdiag \nu$ is thus known up to a classical convolution by a (possibly large) Cauchy random variable. The problem then is turned to a classical deconvolution problem. Notice that, although this is a classical inverse problem, its resolution in general is unstable and difficult. As we shall see, there is much to gain in computing the Stieljes transform closer to the real line and this is actually possible.
\end{itemize}

\subsection{Structure of the paper}
In Section~\ref{section:main}, we state the main results of the paper along three different directions.

\begin{itemize}
\item Computational: We describe our method, and state the theorems it is based on. In particular, there needs to be construction of appropriate contours and reconstruction of a measure from noisy moments. 
\item Statistical: A Cramér-Rao lower bound, showing that in a motivated model, one cannot hope more than a $\frac{1}{n}$ convergence speed.
\end{itemize}

The rest of the paper is devoted to proofs.

In Section~\ref{sec: cp1}, we deal with the optimization aspects of the paper.

In Section~\ref{section:lower_bound}, we tackle the proof of the Cramer-Rao lower bound. In fact, we prove more general statement by first dealing with LAN property and then bounds for a general risk function.

\section{Main results}
\label{section:main}

\subsection{Statistical aspects}

In this section, the goal is to derive a Cram\'er-Rao lower bound for the estimation of the population measure. Because a Cram\'er-Rao lower bound requires a regular model, we will adopt a rather restrictive setup - and only in this part. To that endeavor, fix the number of atoms $q$ and write the population measure:
$$ \mu_\theta = \sum_{k=1}^q w_k \delta_{x_k} ,$$
where $\theta = (x, w) \in C_q \times \Delta_q^{(\R)}$ and
$$ C_q := \left\{ x \in \R_+^q \ | \ 0 < x_1 < x_2 < \dots < x_q \right\} \ ,$$
$$ \Delta_q^{(\R)} := \left\{ w \in (\R_+^*)^q \ | \ \sum_{k=1}^q w_k = 1 \right\} \ ,$$
are respectively the Weyl chamber and simplex of probability measures. The parameter space encoding the measure $\mu_\theta$ is thus the finite dimensional space:
$$ \Theta := C_q \times \Delta_q^{(\R)} \ .$$
Notice that $\Theta$ is an open set inside a space diffeomorphic to $\R^{2q-1}$.

\subsubsection{The model}

For a $\theta \in \Theta$, let us describe the generating law $\P_{n, \theta}$ of our experiment. Given $n$ and $p=p_n$ with 
$$ p_n / n \rightarrow c > 0 \ .$$
One observes a matrix $\X \in M_{p,n}(\R)$ generated as :
$$ \X = V^\half \Y $$
with $\Y$ matrix of iid Gaussian entries. That is to say:
$$ \X \X^* = V^\half \W V^\half \ ,$$
with $\W$ being a white Wishart matrix, and $V = V_{n,\theta}$ is a population matrix.

Furthermore let us introduce the following Gibbs measure supported on the discrete simplex
$$ \Delta_{q,n}^{(\N)}
   :=
   \left\{ k \in \N^q \ | \ \sum_{i=1}^q k_i = n \right\} \ ,$$
and given by:
\begin{align*}
 \P\left( N_k^{(n)} = k_i \ , \ 1 \leq i \leq q \right)
   = & 
   \frac{1}{Z_{n,w}}
   \prod_{i=1}^q \left( \frac{k_i}{n} \right)^{\beta n^2 w_i} \\
   = & 
   \frac{1}{Z_{n,w}}
   \exp\left( \beta n^2 \sum_{i=1}^q w_i \log \frac{k_i}{n} \right) \ .
\end{align*}

Notice that when comparing the entropy
$$
   \beta n^2 \sum_{i=1}^q w_i \log \frac{k_i}{n}
$$
to that of a multinomial distribution:
$$
   n \sum_{i=1}^q \frac{k_i}{n} \log w_i \ ,
$$
we notice that the speed is increased from $n$ to $n^2$. Moreover, there is also an inversion between the parameter $w$ and the estimator $k/n$. 

\begin{assumption}[For Cramér-Rao bounds only]
The population matrix has uniform eigenvectors and independent spectrum so that:
$$ \mu_V = \sum_{i=1}^q \frac{N_k^{(n)}}{n} \delta_{x_k}$$
That is to say
$$ V_{n, \theta} = U^* D^V_{n, \theta} U ,$$
with $U$ Haar distributed on the orthogonal group and $D^V_{n, \theta}$ diagonal with correlated coefficients sampled according to the underdispersed measure $N$.
\end{assumption}

\begin{remark}[Ill-advised choices]
Here are two natural but ill-advised choices. The measure introduced corrects their defects.

{\bf RMT scaling: } Choosing the entries $D^V_{n, \theta}$ as iid sampled according to the measure $\mu_\theta$ gives a spectral measure:
$$
  \mu_V = \sum_{k=1}^q \frac{N_k}{n}\delta_{x_k} \ ,
$$
where $N^{(n)} = \left( N_k \ ; \ 1 \leq k \leq q \right)$ is a multinomial distribution. Because of the independent nature of the sampling, $N_k$'s will have Gaussian fluctuations at scale $n^\half$ -- which is much too large for RMT regime.

{\bf Regularity of the model: } Another common choice is to take $V$ to have spectral measure
$$ 
   \mu_{V} = \sum_{k=1}^q \frac{\lfloor w_k p \rfloor + \Oc(1)}{p}\delta_{x_k}
$$
so that the spectral measure of $V_{n, \theta}$ converges to $\mu_P$. But this choice does not work. The integer floor operation brings too much irregularity to the map $\theta \mapsto \P_{n,\theta}$. Such a model is not regular and thus not amenable to an analysis using likelihood.
\end{remark}

It is well-known \cite[\textsection 7.2]{anderson2003introduction}
that the the law of $\X$ has probability density:
$$ 
   \frac{d \P_{n, \theta}}
        {d \lambda_{M_p(\R)} }(\X)
   =
   \frac{
   \left| \X \right|^{\frac{n-p-1}{2}}
   \exp\left( - \half \tr V_\theta^{-1} \X \right)
   }
   {
   	2^{\frac{np}{2}}
   	\left| V\right|^{\frac{n}{2}}
   	\Gamma_p(\frac{n}{2}) \ ,
   }
$$
where $\Gamma_p$ is the multivariate Gamma function, $\left| \cdot \right|$ is the determinant and $\lambda_{M_p(\R)}$ is the natural Lebesgue measure on symmetric matrices.

\subsubsection{The statement} The setup of limit experiments requires the computation of non-trivial limits for the log-likehood:
$$ \log \frac{d \P_{n, \theta + \varepsilon_n v}}
             {d \P_{n, \theta}} (X)
$$
where $\varepsilon_n$ is an appropriate rescaling speed and $v$ is a deformation in the space $\Theta$.

We can now state the main statistical theorem of this paper.

\begin{thm}
\label{thm:cramer_rao}
Let $\widehat{\theta}_n$ be an estimator which is:
\begin{itemize}
\item regular in the sense of \cite[\textsection 8.5]{vandervaart}
\item asymptotically non-biased.
\end{itemize}
Then, there is a constant $C>0$ such that we have the Cramér-Rao bound:
\begin{align*}
\liminf_{n \rightarrow \infty} 
n \Var \widehat{\theta}_n 
\geq C \ .
\end{align*}
\end{thm}
\begin{proof}[Sketch of proof]
Here we follow the following strategy of proof.
\begin{itemize}
\item We prove a Local Asymptotic Normality (LAN) result in Theorem \ref{thm:LAN}.
\item We invoke the local asymptotic minimax theorem \cite[Theorem 8.11]{vandervaart} -- see also \cite{ibragimov}.
\end{itemize}
\end{proof}




\subsection{Numerical experiments}

In this Subsection, we report on the numerical experiments. Our technique is as described above and an implementation is provided in the github repository

\url{https://github.com/mauricio-velasco/free_deconvolution}

\subsubsection{Description of the other techniques:}
We will describe here the methods mentioned above. Recall here that we observe  a matrix $\mathbb{X}$ satisfying $\mathbb{X}=V^{\frac{1}{2}} \mathbb{Y}$  with   $\mathbb{Y}$ matrix with i.i.d Gaussian entries.

\medskip

{ \bf The convex optimization technique of El Karoui \cite{elkaroui2008}: }
Define 
\begin{align} \label{eq:nu_definition}
    \nu_{n}(z)= \frac{n-p}{nz} + \frac{1}{n} \mathrm{trace}((\mathbb{X}\mathbb{X}^* -z\id_p)^{-1}).
\end{align}

Suppose that the spectral measure of $V$ can be approximated by 
 $$\sum_{k=1}^K w_k\delta_{x_k} $$
Then when $n$ go to infinity  $\nu_{n}$ and under reasonable conditions \cite[Theorem 1 and Subsection 3.2.1]{elkaroui2008} should satisfy the following approximation
 $$0\simeq	\frac{1}{\nu_n(z)}+ z - \frac{p}{n}\sum_{k=1}^K \frac{w_kx_k}{1+x_k\nu_n{(z)}} .$$
 The idea of El Karoui is to choose some grids of values   $(z_j)_{1\le j\le J}$ and $(x_j)_{1\le j\le K}$
 to minimize the norm of the error vector 
 $e=(e_j)_{1\le j\le J}$
 where 
 $$e_j=	\frac{1}{\nu_n(z_j)}+ z_j - \frac{p}{n}\sum_{k=1}^K \frac{w_kx_k}{1+x_k\nu_n{(z_j)}}$$
 under the constraints $ \sum _kw_k=1$ and for all $k, w_k\ge 0$. El Karoui considered the $\ell^p(\R^J)$ norms for $p=1, 2$ and $\infty$. With fixed $x_j$'s and variable unknown $w_k$'s, we have a convex optimization problem under constrains -- which is theoretically appealing.

In practice, the choice of $z_j$'s is crucial, El Karoui, suggested choosing $\nu_n(z_j)$'s first and inverting \eqref{eq:nu_definition} to find the $z_j$'s. Even if El Karoui detailed only the case where the measures are atomic, he suggested doing a similar work with measures having a piecewise affine density. In practice, the change of the dictionary of measures does not seem to affect the result very much.

Cons of the method:
\begin{itemize}
\item No canonical choice and instability when finding $z_j$ as a function of $\nu_n(z_j)$.
\item Stability issue: Small errors in the $e_j$'s do not guarantee proximity of the measures.
\end{itemize}

{\bf The subordination method of Arizmendi et al. \cite{ATV20}:} They use a subordination method for both additive and multiplicative deconvolution. Only the multiplicative case is of interest in the context of estimating covariance matrices. Suppose that $\mu_1\boxtimes\mu_2=\mu_3$ and define the $F-$transform of some measure by $F_\mu (z)=\frac{1}{G_\mu(z)}$.
Under reasonable conditions, \cite[Thm 1.4]{ATV20} proved the existence of some a $\sigma$ that can be computed explicitly, and of a function $\omega_3$ defined on $\mathbb{C}_\sigma :=\{ z \in \C : \Im(z)\ge \sigma\}$ such that
$$
F_{\mu_2}(z) =F_{\mu_3}[\omega_3(z)]z\omega_3(z)^{-1}.
$$

In addition, the iterations of $T_z$,  $T_z^{\circ n}(w)$ converges  to $\omega_3(z)$ when $w$  is in a neighborhood of $z$. Here $T_z(w) := zh_1(h_3(w)^{-1}z^{-1}), \, h_1(w)=w-F_{\mu_1}(w)$ and $h_3(w)=w^{-2}(w-F_{\mu_3}(w)).$

Now that $F_{\mu_2}$ is computable on $\C_\sigma$, we to go from $F_{\mu_2}$ to $\mu_2$. As such, one needs to perform a classical additive deconvolution by a Cauchy measure because
$$
F_{\mu_2}(x + i\sigma)
=
1/G_{\mu_2}(x + i\sigma) \ ,
$$
and
$$ 
- \Im G_{\mu_2}(x + i\sigma)
=
\int_{\mathbb{R}} \frac{\sigma  }{\pi ((x-y)^2+\sigma^2)} \mu_2(dy)
=
\mu_2 * \Lc( \sigma \Cc )(x) \ ,
$$
where $\Cc$ is a standard Cauchy variable.

Here we arrive at the crux of \cite{ATV20}'s method: the non-linear free deconvolution problem has now been turned into a classical deconvolution problem. When $\mu$ is atomic and by an additional discretization (for values of $x$), $\mu$ can be approximated by a solving of an inverse linear problem of type $KU=V$, where $K$ is the classical convolution operator. In \cite{ATV20}, they suggest a Tychonov regularization i.e. to minimize $\|KU-V\|_2+\alpha^2 \|U\|_2$ for some parameter $\alpha$. The reader familiar with statistics can recognize that Tychonov regularization is nothing but the usual Ridge regression.

Cons of the method:
\begin{itemize}
\item Subordination works for relatively high values of $\sigma$. Theoretical lower bounds are very poor. Manual tuning is necessary.
\item Classical deconvolution is an inverse problem. Although it is an extremely classical problem from signal processing, avoiding that is better.
\end{itemize}

\medskip

\subsubsection{Description of scenarios:}
We will test the algorithm on 5 different scenarios.  We will be inspired by the three scenarios given in \cite{elkaroui2008}. In fact, Scenarios 1, 2.1 and 3 are the same as in \cite{elkaroui2008} but we added some modifications of the scenarios to compare the different methods. 

\begin{itemize}
    \item Scenario 1:
    When $\mathbb{X}\mathbb{X}^*$ is a Wishart matrix ($V=\id$).
    \item Scenario 2.1:
    Half of the eigenvalues of $\Sigma$ are $1$, the other half is $2$ and  $(p/n=0.2)$
    \item Scenario 2.2:
    Half of the eigenvalues of $\Sigma$ are $1$, the other half is $1.2$    $(p/n=1)$ 
    \item Scenario 2.3:  The spectral measure of $\Sigma$ is $\frac{ \delta_1+\delta_2+\delta_5+\delta_6+\delta_{8}}{5} $
    
     \item Scenario 3: $\Sigma$ is a Toeplitz matrix with entries $\Sigma_{i,j} = 0.3^{|i-j|}$. 
\end{itemize}

\subsubsection{Give plots}
\begin{itemize}
\item Deconvolution results
We plot the Wasserstein $1$ distance between the ground truth and the estimated spectral measure for different values of $n$. 

\item Performance speeds
\end{itemize}

\subsection{Further comments}

\ 

\medskip

{\bf More applications: } We started with the problem of estimating covariance matrices which is central in the the field of statistics. Indeed, the covariance matrix is the first input of many methods. As such, the relevance of free deconvolution goes beyond for:

\begin{itemize}
\item Outlier detection, PCA. 
\item Another field of application is radio signals (MIMO: Multiple-Input and Multiple-Output), where nowadays networks are formed by a large body of heterogenous antennas and receivers.
\end{itemize}

\medskip

{\bf Comparison to classical deconvolution: } Classical deconvolution is standard topic in signal processing. In that setup, one is dealing with a {\it linear} inverse problem with many inherent instabilities due to high frequencies. Also, the noise level is often unknown. In the case of free deconvolution, the dependence between known and unknown measure is {\it non-linear} and noise level is actually known. As such, we feel that the comparison between classical and free deconvolution can only be fruitful at the level of analogies.

\medskip

{\bf The class of REE: } The above results focus on spectra and ignore completely the matter of eigenvectors. However, if one desires to construct a full matrix $\widehat{\Sigma}$, this can be done at little expense by restricting to the class of Rotation Equivariant Estimators (REE) where eigenvectors are directly specified from the observation.

\begin{definition}[REE]
An estimator $\widehat{\Sigma} = \widehat{\Sigma}\left( \X \right)$ is part of the REE class when it has the property that for all orthogonal matrices $O \in O_n(\R)$:
\begin{align}
\label{def:REE}
\widehat{\Sigma}\left( O \X \right)
& =
O \widehat{\Sigma}\left( \X \right) O^* \ .
\end{align}
In particular, it is natural to diagonalize $\X \X^* = O D O^*$ and consider the estimator
$$
\widehat{\Sigma}\left( \X \right)
=
O \Lambda O^*
$$
where $\Lambda$ is a spectrum to be determined, as a function of the spectrum of $\X$.
\end{definition}

Other papers in the literature refer to this property as Rotation Invariant Estimators \cite{benaych2019optimal, benaych2022short}. In our opinion, the denomination is incorrect since Eq. \eqref{def:REE} does not express an invariance property but rather an equivariance or a covariance property. Furthermore, any refinements outside the class of REE needs to be motivated and is out of the scope of the current paper.

\section{Computational aspects to Stieltjes transforms}
\label{sec: cp1}
The usefulness of the Stieltjes transform $G(z)$ of a measure $\mu$ lies in the fact that it yields a convenient encoding of its moments. As the following Lemma shows, this information is contained in the values of $G(z)$ along any contour around the support $K$ of the measure. In this Section we will show that this representation is very useful computationally. 

\begin{lemma}\label{lem: basic} Suppose $\mu$ is a probability measure supported in a compact set $K\subseteq \RR$ and let $\sigma(t)$ be a contour homotopically equivalent in $\CC\setminus K$ to an ellipse surrounding $K$. If $T(z)$ is any polynomial then the following equality holds 
\[\frac{1}{2\pi i}\oint_{\sigma}T(z)G(z)dz= \int_K T(t)d\mu(t).\]
\end{lemma}
\begin{proof} Since the Stieltjes transform is holomorphic outside the support of the measure, the compactness of $K$ makes it holomorphic at infinity where it has the following power series expansion
\[G(1/w):=\int_\RR \frac{1}{1/w-t}d\mu(t) = \sum_{k=0}^{\infty} w^{k+1}\int_\RR t^k d\mu(t)\]
If $T(z)=z^{j}$ for some integer $j$ and $\sigma$ is the given contour then we have
\begin{tiny}
\[\frac{1}{2\pi i}\oint_{\sigma}T(z)G(z)dz = \frac{1}{2\pi i}\oint_{\sigma}T(1/w)G(1/w)(-1/w^2)dw = \frac{1}{2\pi i}\oint_{-\sigma} \left(\sum_{k=0}^{\infty} w^{k-j-1}\int_\RR t^k d\mu(t)\right)dw.\]
\end{tiny}
Since the function is holomorphic at $w=0$ and $-\sigma$ is a positively oriented contour homotopic to a circle around $w=0$ the Cauchy residue Theorem implies that the integral equals its residue $\int_\RR t^j d\mu(t)$ proving the claim.
\end{proof}

\begin{definition} A {\it contour representation} of the Stieltjes transform $G(z)$ of $\mu$ is a pair $(\sigma(t), v(t))$ for $0\leq t\leq 2\pi$ where $\sigma(t)$ is a parametrization of a curve which goes once around $K$ and $v(t):=G(\sigma(t))$ records the values of $G$ at the points of $\sigma$.
\end{definition}
\begin{remark} In implementations we will represent the pair $(\sigma(t),G(\sigma(t))$ with a sufficiently large collection of pairs of complex numbers $\left(\sigma(t_j), G(\sigma(t_j)\right)$, $j=1,\dots, C$ so that the approximation
\[\int_K T(t)d\mu(t) \sim \frac{1}{2\pi i} 
  \sum_{j=1}^C T(\sigma(t_j))G(\sigma(t_j))\left(\sigma(t_{j})-\sigma(t_{j-1})\right)\]
is sufficiently accurate on polynomials of the desired degrees. Our methods for computing contour representations allow us to increase the number $C$ of points and the accuracy of the values $G(\sigma(t_j))$ as needed.
\end{remark}

The following Lemma gives a contour representation for $G(z)$ knowing the function $S(m)$ on some simply connected neighborhood $U$ of the origin yielding a practical inversion procedure,

\begin{lemma}\label{lem: GfromS} Assume $S(m)$ is known in a simply connected open neighborhood $U$ of the origin. If $m(t)$ is any contour homotopic to a circle around the origin which is contained in $U$ then $\left(z(t),v(t)\right)$ is a contour representation of $G(z)$ where
\[z(t):=\frac{1+m(t)}{m(t)S(m(t))} \text{ and } v(t):=z(t)m(t)-1.\]
\end{lemma}
\begin{proof} By definition of $S(m)$ we know that
\[z(t)=\frac{1+m(t)}{m(t)S(m(t))}=\hat{M}^{-1}(m(t))\]
and therefore $M(z(t))=m(t)$. It follows that the values of $G$ along the points of the contour $z(t)$ are given by $G(z(t))=z(t)M(z(t))-1=z(t)m(t)-1$ as claimed. 
\end{proof}

\subsection{Computing Stieltjes transforms of empirical measures}\label{sec: cp2}

In this section we discuss the computation of the inverse transform $S(m)$ for a discrete measure $\mu$ whose support is contained in a compact set $K\subseteq \RR_{>0}$. More precisely for $j=1,\dots, L$ fix positive real numbers $\lambda_j$ and real positive weights $w_j$ with $1=\sum_{j=1}^M w_j$ and let $\mu:=\sum_{j=1}^M w_j\delta_{\lambda_j}$. In this case the Stjeltjes transform is a rational function of $z$ namely
\[G(z):=\sum_{j=1}^L\frac{w_j}{z-\lambda_j}\text{ and } M(z):=-1+\sum_{j=1}^L\frac{w_j z}{z-\lambda_j}.\]

Our first Lemma gives an effective criterion for verifying the existence of the local inverse function $\hat{M}^{-1}(m(t))$

\begin{lemma} Let $U$ be an open and connected set whose boundary is a simple closed curve $\sigma(t)$ for $0\leq t\leq 2\pi$. Let $f:\overline{U}\rightarrow\CC$ be a holomorphic function such that $f'(z)\neq 0$ for every $z\in U$. The following statements are equivalent:
\begin{enumerate}
\item The function $f$ is one-to-one on $\overline{U}$.
\item The function $f$ is one-to-one on the boundary curve $\sigma(t)$ (equivalently, the curve $\tau(t):=f(\sigma(t))$ for $0\leq t\leq 2\pi$ is a simple closed curve).
\end{enumerate}
\end{lemma}
\begin{proof} The implication $(1)\rightarrow (2)$ is immediate because ${\rm im}(\sigma)\subseteq \overline{U}$. $(2)\implies (1)$ We first claim that there is no $z_0\in U$ and $z_1\in \partial U$ with $f(z_0)=f(z_1)$. Otherwise, since $f'(z_0)\neq 0$ the function $f$ is open near $z_0$ and thus there would exist a point $z^*\in U$ with $f(z^*)$ strictly outside the Jordan curve $\tau(t)$. Since $f(z)-f(z^*)$ is meromorphic in $z$ and has no zeroes or poles on $\sigma$ the argument principle implies that 
\[1\leq \frac{1}{2\pi i} \oint_{\sigma} \frac{[f(z)-f(z^*)]'}{f(z)-f(z^*)}dz = \frac{1}{2\pi i} \int_0^{2\pi} \frac{f'(\sigma(t))\sigma'(t)}{f(\sigma(t))-f(z^*)}dt\]
Making the substitution $\tau(t)=f(\sigma(t))$ the last integral iequals
\[ \frac{1}{2\pi i}\int_0^{2\pi} \frac{f'(\sigma(t))\sigma'(t)}{f(\sigma(t))-f(z^*)}dt =  \frac{1}{2\pi i}\int_0^{2\pi} \frac{\tau'(t)}{\tau(t)-f(z^*)} = \frac{1}{2\pi i}\oint_{\tau}\frac{dw}{w-f(z^*)}=0\]
which is equal to zero since $f(z^*)$ is strictly outside $\tau$. This contradiction shows that there is no $z_0\in U$ and $z_1\in \partial U$ with $f(z_0)=f(z_1)$ and the same argument would imply that there is no $z_0\in U$ with $f(z_0)$ strictly outside the Jordan curve $\tau$. We conclude that for every $z_0\in U$ $f(z_0)$ is in the interior of the region enclosed by $\tau$ the only possible failures of injectivity for $f$ could occur at points $z_0\in U$. To prove injectivity we will count the number of zeroes of the function $f(z)-f(z_0)$ for each $z_0\in U$. By the argument principle the number of zeroes is given by
\[\frac{1}{2\pi i}\oint_{\sigma} \frac{(f(z)-f(z_0))'}{f(z)-f(z_0)}dz =\frac{1}{2\pi i}\oint_{\tau}\frac{dw}{w-f(z_0)} =1 \] 
where the last two equalities follow from the change of variables $\tau(t):=f(\sigma(t))$ and because we have shown that $f(z_0)$ lies in the interior of the region bounded by $\tau$. This equality proves the injectivity of $f$ on $\overline{U}$ as claimed.
\end{proof}

The following Lemma summarizes the basic properties of $M(z)$ for such measures. We use the letter $\mathcal{S}$ to denote the Riemann sphere obtained by endowing the topological sphere $\CC\cup \{\infty\}$ with the unique complex structure that extends the usual one on $\CC$.

\begin{lemma} The following statements hold for all $(\lambda_1,\dots, \lambda_L, w_1,\dots, w_L) \in K^L\times \Delta_{L}$ except for a set of Lebesgue measure zero.
\begin{enumerate}
\item The map $M:\CC\setminus \bigcup_{j=1}^L \{\lambda_j\} \rightarrow \CC$ extends to a unique holomorphic map $M: \mathcal{S}\rightarrow \mathcal{S}$.
\item The map $M$ has degree $L$ and has $2(L-1)$ ramification points (i.e. points $z$ with $M'(z)=0$) with distinct images. The images of these points under $M$, called {\it branch points of $M$}, consist of $L-1$ distinct conjugate pairs in $\CC$.
\item If we denote by $p_1,\dots, p_{L-1}$ the elements of each conjugate pair of branch points with positive imaginary part and we let $U$ be the complement of the vertical lines joining each $p_j$ to infinity and their conjugates then $U$ is simply connected and there is a unique holomorphic function $\hat{M}^{-1}: U\rightarrow \mathcal{S}$ such that $\hat{M}^{-1}(0)=\infty$ and $M(\hat{M}^{-1}(u))=u$ for every $u\in U$.  
\end{enumerate}
\end{lemma}
\begin{proof}
\end{proof}

The previous Lemma suggests an algorithmic construction for $S(m)$ at a given complex number $m$ via the following two steps:
\begin{enumerate}
\item {\it Construct all branch points $p_1,\overline{p_1},\dots, p_{(L-1)}, \overline{p_{(L-1)}}$ of $M$ and define the simply connected domain $U$}. We do this by first finding all ramification points $q_1,\dots, q_{2(L-1)}$ via solving the equation $M'(z)=0$ and letting the $p_i$ be their images under $M$ split into conjugate pairs. Define the simply connected region $U$ as above, namely letting $U$ be the complement of the vertical lines joining each $p_j$ to infinity and their conjugates.
\item {\it Evaluate $\hat{M}^{-1}(m)$ by path lifting}. Choose any path $m(t)$ with $m(0)=0$ and $m(1)=m$ entirely contained in $U$. The value of $\hat{M}^{-1}(m)$ is then uniquely determined by lifting the path $m(t)$ along solutions of the equation $M(w(t))=m(t)$ with the initial condition $w(0)=\infty$. The existence and uniqueness of this lift is immediate from the fact that $M$ is a covering space map away from the inverse image of its branch points. 
\item Evaluate $S(m)$ via $S(m):=(1+m)/m\hat{M}^{-1}(m)$
\end{enumerate}

Carrying out the first two steps above requires specialized algorithmic tools which merit a more precise description, namely:

\begin{enumerate}
\item The construction of all branch points requires finding all solutions of the equation $M'(z)=0$. The difficulty lies in guaranteeing that we have found {\it all} solutions

\item The accurate lifting of paths requires a combination of two ideas, namely homotopy methods and Newton iterations. This idea comes from the extensive literature on numerical algebraic geometry~\cite{Bertini} and can provide extremely accurate estimations of the values of $S$. Concretely, given a path $m(t)$ the computation of the lifted path $w(t)$ proceeds in two steps:
\begin{enumerate}
\item {\it Initial approximation.} Given a step-size $h$ we wish to compute an initial guess $\widetilde{w(t+h)}$. We do so by solving the equation $M(w(t))=m(t)$ up to first order and discretizing, leading to the formula
\[\widetilde{w(t+h)} = w(t) + \frac{m(t+h)-m(t)}{M'(w(t))}\]
\item {\it Newton refinement.} We refine the initial approximation $\widetilde{w(t+h)}$ of $w(t+h)$ via Newton's method applied to the equation $M(w(t+h))=m(t+h)$. We are led to the iterative scheme:
\[
\begin{cases}
w_0 = \widetilde{w(t+h)}\\
w_{n+1} = w_n+a_n\text{ where $a_n:=\frac{m(t+h)-M(w_n)}{M'(w_n)}$}
\end{cases}
\]
We define $w(t+h):=w_n$ for sufficiently large $n$. 
\end{enumerate}
\begin{remark}
Recall that Newton's method is quadratically convergent when the initial value $w_0$ is in the basin of attraction of the true solution and that this condition holds when $w_0$ is computed in the first step for sufficiently small $h$.
\end{remark}

\end{enumerate}

\subsection{Estimating spectral densities}

Combining all our previous work we are now ready to construct our estimation of the pectral density. Recall that we only observe the empirical spectral measure $\mu_n$ of $\frac1n \X \X^*$ and wish to construct an estimator $\widehat{\nu}_n$ of the population spectral measure $\nu$. We do so indirectly by estimating $S_{\nu}$. The main result of this Section is the following

\begin{thm}\label{thm: estimate_S}  If $T_n(m):=S_{\mu_n}(m)/S_{MP}(m)$ then the following statements hold: 
\begin{enumerate}
\item $T_n(m)$ converges to $S_{\nu}(m)$ uniformly on compact subsets of the origin in $\CC$ and 
\item the rate of convergence satisfies $\|T_n-S_{\nu}\|\sim O(1/n)$
\end{enumerate}
\end{thm}

The previous Theorem suggests a procedure for carrying out the estimation of the spectral measure $\nu$
\begin{enumerate}
\item Compute $T_n(m):=S_{\mu_n}(m)/S_{MP}(m)$ where the numerator is approximated as in Section~\ref{sec: cp2}.
\item Use $T_n(m)$ to obtain a contour representation of an approximation of $G_{\nu}(m)$ as in Lemma~\ref{lem: GfromS}.
\item Build the estimator measure $\hat{\nu}$ via recovery of the measure $\nu$ from its contour representation, that is from approximate knowledge of its moments obtained via Lemma~\ref{lem: basic}.
\end{enumerate}
Depending on the nature of the spectral measure $\nu$ we will use different recovery mechanisms, which will be discussed in the remainder of this section.


\section{Cram\'er-Rao lower bound}
\label{section:lower_bound}

Let us introduce a few notations. It is useful to write
$$ \left( x_1(\theta), \dots, x_q(\theta) \right)$$
$$ \left( w_1(\theta), \dots, w_q(\theta) \right)$$
for respectively the support and the weights associated to a parameter $\theta$. 

Also, by separating the deformation $v$ as $v=(g,h)=v_g + v_h$, into a deformation $h \in \R^q$ along the support and $g \in \R^q$ along the weights, we have:
$$ x_k( \theta + \varepsilon_n v ) = x_k(\theta) + \varepsilon_n h_k \ ,$$
$$ w_k( \theta + \varepsilon_n v ) = w_k(\theta) + \varepsilon_n g_k \ .$$
Naturally, we need to take $\sum g_k = 0$ in order to remain tangeant to the simplex, while $h$ is free. 

Now, we can state the main result of this section.

\begin{thm}[Local Asymptotic Normality (LAN)]
\label{thm:LAN}

For the speed $\varepsilon_n = \frac{1}{n}$, we have the following limit in law:
\begin{align*}
  & 
\lim_{n \rightarrow \infty} \log \frac{d \P_{n, \theta + \varepsilon_n v}}
             {d \P_{n, \theta}} (\X) \\
= & \
-\frac{1}{4} c
\sum_{k=1}^q 
            w_k(\theta)
            \frac{h_k^2}{x_k(\theta)^2}      
+\half c^\half
 \sum_{j=1}^q
 \Nc_k
 w_k(\theta)
 \frac{h_k}
      {x_k(\theta)} \ ,
\end{align*}
under the reference distribution $d \P_{n, \theta}$.
\end{thm}

\subsection{Proof of the LAN Theorem \ref{thm:LAN}}

Let us start with some notations. The i.i.d. random variables $m_i \in \{ 1, 2, \dots, q\}$ are the outcomes of the $n$ modalities such that:
$$ [D^V]_{i,i} = x_{m_i} \ .$$
Associated to that is the Gibbs-type measure $N = \left( N_1, \dots, N_q \right)$ given by:
$$ N_k := \Card \left\{ 1 \leq i \leq n \ | \ m_i = k \right\} \ .$$
The (discrete) part of the log-likelihood is:
$$ \beta n^2 \sum_{k=1}^q w_k \log \frac{N_k}{n} \ .$$ 

From the formula of the density
\begin{align*}
  & 
\log \frac{d \P_{n, \theta + \varepsilon_n v}}
             {d \P_{n, \theta}} (\X) \\
= & 
\beta n^2 \sum_{k=1}^q \left( w_k(\theta + \varepsilon_n v) - w_k(\theta) \right) \log \frac{N_k}{n}\\
  & \quad 
-\frac{n}{2} \log \frac{\left| V_{n, \theta + \varepsilon_n v} \right|}
                       {\left| V_{n, \theta} \right|}
-\half \tr \left( ( V_{n, \theta + \varepsilon_n v}^{-1} 
                  - V_{n, \theta}^{-1} ) \X \right) \ .
\end{align*}

We shall now analyze the three terms.

\medskip

{\bf Step 1: Parametrizing $\Theta$ and the first two terms.}

Thanks to the convenient notations introduced before the statement of the theorem, we can write:
\begin{align*} 
\log \left| V_{n, \theta} \right|
= & \log \left| D^V_{n, \theta} \right| \\
= & \sum_{j=1}^p \log \left( D^V_{n, \theta} \right)_{j,j} \\
= & \sum_{k=1}^q p_k \log x_k(\theta) \ .
\end{align*}
As such, we have:
\begin{align*} 
  & 
-\frac{n}{2} \log \frac{\left| V_{n, \theta + \varepsilon_n v} \right|}
                       {\left| V_{n, \theta} \right|} \\
= & -\frac{n}{2} 
     \left( \sum_{k=1}^q 
            p_k \log x_k(\theta+\varepsilon_n v) 
            -
            p_k \log x_k(\theta) 
     \right) \\ .
= & -\frac{n}{2} 
		    \sum_{k=1}^q 
            p_k
            \log\left( 1 + \frac{\varepsilon_n h_k}{x_k(\theta)} \right)
            \ .
\end{align*}

Therefore, the two first terms are:
\begin{align*}
& 
\sum_{k=1}^q p_k \log \frac{w_k(\theta + \varepsilon_n v)}{w_k(\theta)}
-\frac{n}{2} \log \frac{\left| V_{n, \theta + \varepsilon_n v} \right|}
                       {\left| V_{n, \theta} \right|} \\
= & 
\sum_{k=1}^q p_k \log \left( 1 + \frac{\varepsilon_n g_k}{w_k(\theta)} \right)
-\frac{n}{2} 
\sum_{k=1}^q 
            p_k
            \log\left( 1 + \frac{\varepsilon_n h_k}{x_k(\theta)} \right) \ .
\end{align*}

\medskip

{\bf Step 2: Expression in term of population matrices}

Now recall two facts. First $X$ is taken under the law $\P_{n,\theta}$, so that
$$ X = V_{n,\theta}^{\half} \W V_{n,\theta}^{\half} \ .$$
Second a white Wishart matrix with parameters $(p, n)$ can be written as a sum rank 1 projectors:
$$ \W = \sum_{i=1}^n \xi_i \xi_i^* \ ,$$
where the $\xi_i$'s are iid vectors in $\R^p$ with standard Gaussian entries. As such, starting with the cyclic property of the trace, we have for the third term:

\begin{align*}
& 
-\half \tr \left( ( V_{n, \theta}^{\half} 
                                 V_{n, \theta + \varepsilon_n v}^{-1} 
                                 V_{n, \theta}^{\half}
                  - \id ) \W \right) \\
= & 
-\half \sum_{i=1}^n \tr \left( \xi_i^* ( V_{n, \theta}^{\half} 
                                 V_{n, \theta + \varepsilon_n v}^{-1} 
                                 V_{n, \theta}^{\half}
                  - \id ) \xi_i \right) \\
= & 
-\half \sum_{i=1}^n \xi_i^* ( V_{n, \theta}^{\half} 
                                 V_{n, \theta + \varepsilon_n v}^{-1} 
                                 V_{n, \theta}^{\half}
                  - \id ) \xi_i \ .
\end{align*}

Because the population matrices $V_{n, \theta}$ are co-diagonalizable, we can simplify further the quadratic forms in the above expression:
\begin{align*}
  & 
-\half \tr \left( ( V_{n, \theta}^{\half} 
                                 V_{n, \theta + \varepsilon_n v}^{-1} 
                                 V_{n, \theta}^{\half}
                  - \id ) \W \right) \\
= & 
-\half \sum_{i=1}^n \left( U \xi_i \right)^* 
                    \left( \frac{D^V_{n, \theta}}
								{ D^V_{n, \theta + \varepsilon_n v} }
                  - \id \right) U \xi_i \\
= & 
-\half \sum_{i=1}^n \sum_{j=1}^p
 \left[ (U \xi_i)^j \right]^2 \left( \frac{(D^V_{n, \theta})_{j,j}}
                         {(D^V_{n, \theta + \varepsilon_n v})_{j,j}} - 1
             \right) \\
= & 
-\half \sum_{j=1}^p
 \chi_{n,j}^2 \left( \frac{(D^V_{n, \theta})_{j,j}}
                         {(D^V_{n, \theta + \varepsilon_n v})_{j,j}} - 1
             \right) \ ,
\end{align*}
where $\chi_{n,j}^2$ are iid $\chi^2$ distributions with parameter $n$.

Now, we need to group terms in terms of the multinomial distribution $(p_1, \dots, p_k)$:
\begin{align*}
  & 
-\half \tr \left( ( V_{n, \theta}^{\half} 
                                 V_{n, \theta + \varepsilon_n v}^{-1} 
                                 V_{n, \theta}^{\half}
                  - \id ) \W \right) \\
= & 
-\half \sum_{j=1}^q
 \left( \sum_{j: m_j=k}\chi_{n,j}^2 \right)
 \left( \frac{x_k(\theta)}
             {x_k(\theta + \varepsilon_n v)} - 1
             \right) \\
= & 
-\half \sum_{j=1}^q
 \chi_{n p_k,k}^2
 \left( \frac{x_k(\theta)}
             {x_k(\theta + \varepsilon_n v)} - 1
             \right) \\
= & 
 \half \sum_{j=1}^q
 \chi_{n p_k,k}^2
 \frac{\varepsilon_n h_k}
      {x_k(\theta + \varepsilon_n v)} \ ,
\end{align*}
where $\chi^2_{\cdot, k}$ are again independent $\chi^2$ random variables, with the specified parameter.
 
\medskip

{\bf Step 3: Grouping terms.} In the end, we have:

\begin{align*}
  & 
\log \frac{d \P_{n, \theta + \varepsilon_n v}}
             {d \P_{n, \theta}} (\X) \\
= & 
\beta n^2 \sum_{k=1}^q \varepsilon_n g_k \log \frac{N_k}{n}
-\frac{n}{2} 
\sum_{k=1}^q 
            p_k
            \log\left( 1 + \frac{\varepsilon_n h_k}{x_k(\theta)} \right) \\
  & \quad 
+\half \sum_{j=1}^q
 \chi_{n p_k,k}^2
 \frac{\varepsilon_n h_k}
      {x_k(\theta + \varepsilon_n v)} \\
= & 
\beta n^2 \sum_{k=1}^q \varepsilon_n g_k \log \frac{N_k}{n}
-\frac{n}{2} 
\sum_{k=1}^q 
            p_k
            \left( 
            \log\left( 1 + \frac{\varepsilon_n h_k}{x_k(\theta)} \right) 
            -            
            \frac{\varepsilon_n h_k}
		         {x_k(\theta + \varepsilon_n v)}
            \right)            
            \\
  & \quad 
+\half \sum_{j=1}^q
 \left( \chi_{n p_k,k}^2 - n p_k \right)
 \frac{\varepsilon_n h_k}
      {x_k(\theta + \varepsilon_n v)} \ .
\end{align*}

Because
\begin{align*}
 &  \log\left( 1 + \frac{\varepsilon_n h_k}{x_k(\theta)} \right)
             - 
             \frac{\varepsilon_n h_k}
                  {x_k(\theta)+\varepsilon_n h_k} \\
= & \Oc( \varepsilon_n^3 )
+ \frac{\varepsilon_n h_k}{x_k(\theta)}
- \frac{\varepsilon_n^2 h_k^2}{2 x_k(\theta)^2}
- \frac{\varepsilon_n h_k}
       {x_k(\theta)}
  \left( 1 - \frac{\varepsilon_n h_k}{x_k(\theta)} 
           + \Oc(\varepsilon_n^2)  
  \right) \\ 
= & \Oc( \varepsilon_n^3 )
+ \frac{\varepsilon_n^2 h_k^2}{2 x_k(\theta)^2} \ ,
\end{align*}
and:
\begin{align*}
    \log\left( 1 + \frac{\varepsilon_n g_k}{w_k(\theta)} \right)
= & \ \Oc( \varepsilon_n^2 )
+ \frac{\varepsilon_n g_k}{w_k(\theta)} \ ,
\end{align*}
we obtain:
\begin{align*}
  & 
\log \frac{d \P_{n, \theta + \varepsilon_n v}}
             {d \P_{n, \theta}} (\X) \\
= & \ \Oc\left( p \varepsilon_n^2 + n p \varepsilon_n^3 \right) +
\sum_{k=1}^q p_k \frac{\varepsilon_n g_k}{w_k(\theta)}
-\frac{n}{2} 
\sum_{k=1}^q 
            p_k
            \frac{\varepsilon_n^2 h_k^2}{2 x_k(\theta)^2}      
            \\
  & \quad 
+\half \sum_{j=1}^q
 \left( \chi_{n p_k,k}^2 - n p_k \right)
 \frac{\varepsilon_n h_k}
      {x_k(\theta + \varepsilon_n v)} \\
= & \ \Oc\left( p \varepsilon_n^2 + n p \varepsilon_n^3 \right) +
\varepsilon_n p
\sum_{k=1}^q \frac{p_k}{p} \frac{g_k}{w_k(\theta)}
-\frac{np}{4} \varepsilon_n^2
\sum_{k=1}^q 
            \frac{p_k}{p}
            \frac{h_k^2}{x_k(\theta)^2}      
            \\
  & \quad 
+\half \varepsilon_n n^\half p^\half
 \sum_{j=1}^q
 \frac{\chi_{n p_k,k}^2 - n p_k}
      { \sqrt{n p_k} }
 \frac{p_k}{p}
 \frac{h_k}
      {x_k(\theta + \varepsilon_n v)} \ .
\end{align*}

\medskip

{\bf Step 4: Limit.} Recall the limits:
$$ \lim_{n \rightarrow \infty} p_n/n = c \ ,$$
$$ \lim_{n \rightarrow \infty} p_k/p = w_k(\theta) \ ,$$
$$ \lim_{n \rightarrow \infty} \frac{\chi^2_\alpha - \alpha}{\alpha^\half}
   = \Nc_k .$$
The first one is a hypothesis, the second one is the law of large numbers and the third is a standard limit in law.

Thanks to the final expression in Step 3, we see that we have convergence for $\varepsilon_n = \frac{1}{n}$. The limit is:
\begin{align*}
  & 
\lim_{n \rightarrow \infty} \log \frac{d \P_{n, \theta + \varepsilon_n v}}
             {d \P_{n, \theta}} (\X) \\
= & \
c
\sum_{k=1}^q w_k(\theta) \frac{g_k}{w_k(\theta)}
-\frac{1}{4} c
\sum_{k=1}^q 
            w_k(\theta)
            \frac{h_k^2}{x_k(\theta)^2}      
+\half c^\half
 \sum_{j=1}^q
 \Nc_k
 w_k(\theta)
 \frac{h_k}
      {x_k(\theta)} \\
= & \
-\frac{1}{4} c
\sum_{k=1}^q 
            w_k(\theta)
            \frac{h_k^2}{x_k(\theta)^2}      
+\half c^\half
 \sum_{j=1}^q
 \Nc_k
 w_k(\theta)
 \frac{h_k}
      {x_k(\theta)} \ .
\end{align*}
This is the announced result.




\appendix

\section{Generalities on Riemann surfaces and Markov-Krein}


Since we are interested in a computational understanding of what is happening, we specialize to a measure of the form:

$$ \mu_d := \sum_{j=1}^d w_j \delta_{x_j}(dx) $$
where $w_j$ are non-negative weights on the simplex and $x_j$ give the support. In this particular case, $M: \C \backslash \R \rightarrow \C$ is a rational function. Naturally, it extends to a map from the Riemann sphere $\S$ to itself.

Classically, this gives rise to a ramified $d$-covering of $\S$ of as follows. The critical points are the points $z \in \S$ where $M$ fails to be locally invertible
$$ Z := \left\{ z \in \S \ | \ M'(z) = 0 \right\} \ ,$$
while the ramification/branch points are the image points $m \in \S$:
$$ M(Z) := \left\{ m \in \S \ | \ \exists z \in Z, \ m = M(z) \right\} \ .$$
The degree is $d$ because a generic $m \in \S$ has $d$ pre-images via $n$. Moreover, $\Card Z = 2(d-1)$ and in fact, critical points are complex conjugates as roots of a polynomial of degree $2(d-1)$.


\bigskip

{\bf Zeros of the first and second kind. } In this paragraph, let us borrow some notations and terminology from OPRL (Orthogonal Polynomials on the Real Line). Our working measure $\mu_d$, in practice and in the theory of FPT, is the approximation of a reference measure $\mu_\infty$. In the context of quadrature approximation, recall that:
$$ G_{\mu_d}(z) = \int_\R \frac{\mu_d(dx)}{z-x} = \frac{Q_d(z)}{P_d(z)} \ ,$$
$$ M_{\mu_d}(z) = \frac{z Q_d(z)}{P_d(z)} - 1 \ .$$
The roots of $P_d$ and $Q_d$, respectively denoted by
$$ \left( x_j \ ; \ 1 \leq j \leq d \right)
   \textrm{ and }
   \left( y_j \ ; \ 1 \leq j \leq d-1 \right) \ ,
$$
are called zeros of the first and second kind. 

The Markov-Krein transform consists in defining the measure:
$$ \nu_d = \delta_0 + \sum_{j=1}^{d-1} \delta_{y_j} - \sum_{j=1}^{d} \delta_{x_j} \ ,$$
and writing
$$ M_{\mu_d}(z) = \exp\left( \int_\R \nu_d(dx) \log\left( z-x \right) \right) - 1 \ .
$$

The following Lemma shows that critical points are zeros of the Cauchy-Stieljes transform of $\nu_d$, which is the Markov-Krein transform of $\mu_d$ (plus a $\delta_0$). 

\begin{lemma}
The following statements are equivalents:
$$ M'(z) = -\int_\R \frac{\mu_d(dx)}{(z-x)^2} = -\sum_{j=1}^d \frac{w_j}{(z-x_j)^2} = 0 $$
$$ 
   \int_\R \frac{\nu_d(dx)}{z-x} = \frac{1}{z} + \sum_{j=1}^{d-1} \frac{1}{z-y_j} - \sum_{j=1}^{d} \frac{1}{z-x_j} = 0 \ .
$$
\end{lemma}
\begin{proof}
Write $M(z) = e^{F(z)}$ with $F(z) = \int_\R \nu_d(dx) \log\left( z-x \right)$. As such $M'(z) = F'(z) e^{F(z)} 0$ if and only if $F'(z) = 0$. We are done by noticing that $F'$ is the required Cauchy-Stieljes transform.
\end{proof}

\begin{remark}
More than a mere curiosity, it is useful so that every numerical procedure (argument principle, Newton-Raphson etc...) has to be tailored for Cauchy-Stieljes transforms {\it only}.

Also zeros of the second kind can be computed extremely fast by dichotomy.
\end{remark}

\section{Measures and moments}
In this section, we want to discuss how to approximate numerically a  measure knowing its moments to do so, we will recall first  some notions.  
\\ We say that a sequence $(m_n)_{n\geq 0}$ is a  (Hamburger) moment sequence if there exists some  real measure  $\mu$, such that $\forall n\geq0$ $m_n=\int x^nd\mu(x)$. In particular, the following characterizations are well-known (add a reference?). 
\begin{thm} \label{thm:measures}
The following assertions are equivalent : 
\begin{enumerate}
    \item $(m_n)_{n\geq 0}$  is a moment sequence 
    \item The Hankel kernal 
  \[ H=\left(\begin{matrix}
m_0 & m_1 & m_2 & \cdots     \\
m_1 & m_2 & m_3 & \cdots  \\
m_2 & m_3 & m_4 & \cdots  \\
\vdots & \vdots & \vdots & \ddots
\end{matrix}\right)\]
is positive  semi-definite.
\item  In the continuous fraction,  
\[\sum_{n=0}^\infty m_nz^n= \frac{1}{1-\alpha_0 z -\frac{\beta_1z^2}{1-\alpha_1z -\frac{\beta_2z^2}{\ddots}}},\]
the coefficients $\beta_n$ are non-negative for every $n\geq 1$.
\end{enumerate}
Moreover, 
 \begin{itemize}
     \item  $\mathrm{card}(\mathrm{supp}(\mu))=\infty$ if and only if  $H$ is positive definite
     \item  $\mathrm{card}(\mathrm{supp}(\mu))=n_0$ if and only if $\beta_{n}>0$ for any $n<n_0$ and  $\beta_{n_0}=0$. 
 \end{itemize}
\end{thm}

\subsection{An orthogonal polynomial point of view}
Let $\mu$ be a measure (with all moments finite).
Let $m_n(\mu)=\int x^n d_\mu$ and define  the scalar product $\langle  .\rangle _\mu$ by  $\langle  g,f\rangle _{\mu}=\int fg d\mu$. One can define $(p_i)_{i\geq 0}$ the set of unitary orthogonal polynomial associated to $\mu$. i.e.
 \begin{itemize}
     \item $\forall i \in \mathbb{N}, deg(p_i)=i$
     \item $\forall i \in \mathbb{N}, [x^i]p_i(x)=1$
     \item $\forall i\neq j, \langle  p_i,p_j\rangle _\mu=0$.
 \end{itemize}
Not let
 \[ H_n=\left(\begin{matrix}
m_0(\mu) & m_1(\mu) & m_2(\mu) & \cdots  &m_{n-1}(\mu)   \\
m_1(\mu) & m_2(\mu) & m_3(\mu) & \cdots &m_{n}(\mu) \\
\vdots & \vdots & \vdots & \ddots & \vdots \\
 m_{n-1}(\mu)  &m_{n}(\mu) &m_{n+1}(\mu) &\cdots &m_{2n-1}(\mu)
\end{matrix}\right)\]

\begin{proposition}

The Cholesky factorization $H_n=LL^T$ exists and the entries of $L$ are $L_{i,j} = \frac{\langle  p_{j-1},x^{i-1}\rangle _\mu}{\sqrt{\langle  p_{j-1},p_{j-1}\rangle _\mu}}.$
\end{proposition}
\begin{proof}
By construction $H_n$ is symmetric, and Theorem~\ref{thm:measures} guarantees that it is positive semi-definite.  So there exists a unique Lower diagonal Matrix $L_1$ such that $H_n=L_1L_1^T$.  We only need to check that the matrix  $L$ defined in the proposition is lower diagonal and satisfies $H_n=LL^T$.

Since $x^{i-1} \in \mathbb{R}_{i-1}[X]$ the vector space generated by $(p_\ell)_{\ell<i}$ then  ${\langle  p_{j-1},x^{i-1}\rangle _\mu}=0$ as soon as $j>i$ and then $L$ is lower diagonal. 

Moreover,  since $\{ p_k \}_{0\le k \le n-1}$ is an orthogonal basis of $\mathbb{R}_{n-1}[X]$ then  the $(i,j)$ entry of $H_n$ is 
\begin{align*}
    m_{i+j-2}(\mu)= 
\langle  x^{i-1},x^{j-1}\rangle _{\mu} &= 
\sum_{k=0}^{n-1}  \frac{\langle  x^{i-1},p_k\rangle _{\mu}\langle  x^{j-1},p_k\rangle _{\mu} }{\langle  p_k,p_k\rangle _\mu}\\&=\sum_{k=0}^{n-1}  {L_{i,k}L_{j,k}}
\end{align*}
this is exactly the (i, j) entry of the matrix $LL^T$ which concludes the proof.   
\end{proof}

Notice that above argument above works if we replace $H_n$ by any Gram matrix.  

\subsection{Three-terms recurrence}
The orthogonal polynomial defined above have a three terms recurrence equation,
 $$ X p_n = p_{n+1} + a_n p_n + b_{n} p_{n-1}. $$
One way to see it is that $Xp_n-p_{n+1}$ is a polynomial of degree at most {n}. 

and for any $i \le n-2$, $Xp_i \in \mathbb{R}_{n-1}[X]$ and 

$$ \langle  Xp_n-p_{n+1},p_i\rangle _{\mu}=\langle  p_n-Xp_i\rangle _{\mu} +\langle  p_{n+1},p_i\rangle _{\mu}=0. $$

In this case, 
\begin{align*}
    \langle p_n,p_n\rangle _{\mu} = \langle p_n,Xp_{n-1}\rangle _{\mu} +
\langle p_n,p_n-Xp_{n-1}\rangle _{\mu}&=\langle Xp_n,p_{n-1}\rangle _{\mu}\\&= b_n\langle p_{n-1},p_{n-1}\rangle .\end{align*}
In particular, $b_n \geq 0$
and a simple recurrence shows that $\langle p_n,p_n\rangle =\mu(\mathbb{R})^2 \prod_{i=1}^n{b_i}.$ Moreover,

$L_{n,n} = \frac{\langle p_{n-1},x^{n-1}\rangle _\mu}{\sqrt{\langle p_{n-1},p_{n-1}\rangle _\mu}}= \frac{\langle p_{n-1},p_{n-1}\rangle _\mu}{\sqrt{\langle p_{n-1},p_{n-1}\rangle _\mu}} ={\sqrt{\langle p_{n-1},p_{n-1}\rangle _\mu}} $ and 

$$b_n=\frac{\langle p_{n},p_{n}\rangle _\mu}{\langle p_{n-1},p_{n-1}\rangle _\mu}= \frac{L_{n+1,n+1}^2}{L_{n,n}^2}.$$

To recover $a_n$, we have 

\begin{align*}
L_{n+2,n+1}&= \frac{\langle X^{n+1},P_{n}\rangle _\mu}{||p_n||_{\mu}} \\&=
\frac{\langle X^{n},XP_{n}\rangle _\mu}{{||p_n||_{\mu}}}
\\&=
\frac{\langle X^{n},P_{n+1}+a_{n}P_{n}+b_{n}P_{n-1}\rangle _\mu}{{||p_n||_{\mu}}}
\\&=\frac{\langle X^{n},P_{n+1}\rangle _\mu+a_{n}\langle X^{n},P_{n}\rangle _\mu+b_{n}\langle X^{n},P_{n-1}\rangle _\mu}{{||p_n||_{\mu}}}
\\&= a_{n} L_{n+1,n+1} + b_n L_{n+1,n} \frac{{||p_{n-1}||_{\mu}}}{||p_n||_{\mu}} = a_{n} L_{n+1,n+1} +  \frac{L_{n+1,n}}{b_n}.
\end{align*}
 
\subsection{How to recover the measure from Jacobi coefficients?}

Let $J^{(n)}= [J^{(n)}_{i,j}]_{i,j\le n}$ be the Jacobi matrix, which is tridiagonal of size $n$ with diagonal elements equal $J^{(n)}_{i,i}=a_i$ and extradiagonal elements  $J^{(n)}_{i+1,i} = J^{(n)}_{i,i+1}=b_i$. 

We diagonalize the matrix as
$$ J^{(n)}
   =
   V^* \diag\left( \lambda^{(n)}_i \right) V \ ,
$$
where are $\lambda^{(n)}_1\le \lambda^{(n)}_2\le \dots\le\lambda^{(n)}_n$ the set of eigenvalues of $J^{(n)}$ and $v^{(n)}_i$ be the unit eigenvector associated to $\lambda^{(n)}_i$. We have then the following.

\begin{proposition}
If $card(supp(\mu))\le n$ then
$$ \mu(dx) 
   = 
   \sum_{i=1}^n \langle v^{(n)}_i, e_1\rangle ^2 \delta_{ \lambda^{(n)}_i }(dx) \ . $$
\end{proposition}



\bibliographystyle{amsalpha}
\bibliography{FreeStats.bib}

\end{document}